\documentclass[11pt]{amsart}  %<<<1
\usepackage{amssymb,amsmath,amsthm,amscd}
\usepackage{graphicx}
\numberwithin{equation}{section}
\newtheoremstyle{fancy1}{10pt}{10pt}{\itshape}{12pt}{\textsc\bgroup}{.\egroup}{8pt}{
}
\newtheoremstyle{fancy2}{10pt}{10pt}{}{12pt}{\itshape}{.}{8pt}{ }

\theoremstyle{fancy1}

\newtheorem{lem}[equation]{Lemma}
\newtheorem{prop}[equation]{Proposition}

\newtheorem*{thm*}{Theorem}

\newtheorem{main}{Theorem}
\newtheorem*{main*}{Theorem}

\newtheorem*{cor*}{Corollary}
\newtheorem*{prop*}{Proposition}

\newtheorem*{problem*}{Problem}

\theoremstyle{fancy2}

\newtheorem{rem}[equation]{Remark}
\newtheorem*{rems*}{Remarks}

\newtheorem*{rem*}{Remark}

\newtheorem*{example*}{Example}

\newcommand{\cref}[1]{Corollary~\ref{#1}}

\newcommand{\pref}[1]{Proposition~\ref{#1}}

\newcommand{\e}{\epsilon}
\newcommand{\RP}{\mathbb{R\mkern1mu P}}
\newcommand{\CP}{\mathbb{C\mkern1mu P}}

\newcommand{\Sph}{\mathbb{S}}
\newcommand{\Disc}{\mathbb{D}}

\newcommand{\C}{{\mathbb{C}}}
\newcommand{\R}{{\mathbb{R}}}
\newcommand{\Z}{{\mathbb{Z}}}

\newcommand{\QH}{{\mathbb{H}}}

\newcommand{\D}{\ensuremath{\operatorname{D}}}

\newcommand{\Pin}{\ensuremath{\operatorname{Pin}}}

\newcommand{\fg}{{\mathfrak{g}}}
\newcommand{\fk}{{\mathfrak{k}}}
\newcommand{\fh}{{\mathfrak{h}}}
\newcommand{\fm}{{\mathfrak{m}}}
\newcommand{\fn}{{\mathfrak{n}}}

\newcommand{\fp}{{\mathfrak{p}}}

\newcommand{\fl}{{\mathfrak{l}}}
\newcommand{\fso}{{\mathfrak{so}}}

\def\con#1=#2(#3){#1 \equiv #2 \bmod{#3}}

\newcommand{\ml}{\langle}                     % Riemannian metric (left )
\newcommand{\mr}{\rangle}                    % Riemannian metric (right)

\newcommand{\diag}{\ensuremath{\operatorname{diag}}}

\newcommand{\Ad}{\ensuremath{\operatorname{Ad}}}

\renewcommand{\sec}{\ensuremath{\operatorname{sec}}}
\newcommand{\Ric}{\ensuremath{\operatorname{Ric}}}

\DeclareMathOperator{\Id}{Id}

\begin{document}
%\date{\today}

\title{Curvature homogeneous manifolds in dimension 4}

\author{Luigi Verdiani}
\address{University of Firenze}
\email{verdiani@math.unifi.it}
\author{Wolfgang Ziller}
\address{University of Pennsylvania}
\email{wziller@math.upenn.edu}
\thanks{The first named author was supported by Prin and GNSAGA grants.  The second named author was supported by a grant from
the National Science Foundation and by a fellowship from CNPq to support his visit at IMPA}

\begin{abstract}
We classify complete curvature homogeneous metrics on simply connected  four dimensional manifolds which are invariant under a cohomogeneity one action. 
\end{abstract}

\maketitle

Let $M$ be a Riemannian manifold.  $M$ is called curvature homogeneous if, for any points $p,q\in M$, there exists a linear isometry $f: T_pM\to T_q M$ that preserves the curvature tensor, i.e. $f^* R_q=R_p$. In \cite{Si} I. Singer asked the question whether such manifolds are always homogeneous. The first complete counter examples were given by K.Sekigawa and H.Takagi  in \cite{Se,Ta}, later generalized in \cite{KTV}. In these examples the metric depends on several  arbitrary functions of one variable and the curvature tensor is equal to the curvature tensor of the isometric product $\QH^2\times\R^{n-2}$, with a flat metric on $\R^{n-2}$. In \cite{Br} it was shown that, in dimension $3$, there are complete examples which achieve any free group as their fundamental group, which can in addition depend on infinitely many arbitrary functions. There are many other local non-homogeneous examples which are curvature homogeneous, especially in dimension 3, see \cite{BKV} and references therein. We point out though that the only known  compact non-homogeneous examples are the Ferus-Karcher-M\"unzner  isoperimetric hypersurfaces in $\Sph^n$, see \cite{FKM}.

It is natural to look for further examples among the class of cohomogeneity one manifolds, i.e. Riemannian manifolds on which a Lie $G$ acts whose generic orbits are hypersurfaces. One such example was discovered by K.Tsukada, \cite{Ts}. It is a complete metric on a two-dimensional vector bundle over $\RP^2$, the normal bundle of the Veronese surfaces  $\RP^2\subset\CP^2$. The Lie group $SO(3)$ acts on it by cohomogeneity one and the metric is given by 
$$
ds^2=dt^2+e^{2t}\, d\theta_1^2  +e^{-2t}\, d\theta_2^2  +(e^t-e^{-t})^2\, d\theta_3^2
$$
where $t$ is the arc length parameter of a geodesic normal to all orbits, and $d\theta_i$ is the dual of the usual basis on the Lie algebra of $SO(3)$.

We will show that this example is indeed very special. 

\begin{main}
	Let $(M,G)$ be a four dimensional simply connected  cohomogeneity one manifold with $G$ a compact Lie group. Then any smooth complete  curvature homogeneous $G$-invariant metric  is either isometric to a symmetric space, or  to the Tsukada example.
\end{main}
Notice that this includes the case of warped product metrics of the form $dt^2+g_t$ on $\R^4$ where $g_t$ is an arbitrary left invariant metric on $SU(2)$, thus depending on $6$ functions of one variable. 
\smallskip

One can describe the proof as follows. The condition of being curvature homogeneous reduces to an ODE along the normal geodesic in terms of the metric and a $G$-invariant connection. Since this  system of ODE's depends on 9 functions, it is too complicated to solve for the metric directly.
So we first discuss the case of diagonal metrics. In this case we will show  curvature homogeneity implies that the components of the curvature tensor are constant, which will imply that the metric is described by  linear combinations of trigonometric, hyperbolic and linear functions. This gives rise to finitely many algebraic equations. If there exists a singular orbit, one can furthermore use the required smoothness conditions at the singular orbit, which makes it fairly straight forward to solve the  system of equations.  The result  is that the metric is equivariantly isometric to one of the 13  cohomogeneity one actions on the  4-dimensional symmetric spaces, or to the Tsukada example.

 For a general metric we first prove that, using an equivariant diffeomorphism,  the matrix of functions describing the metric can be partially diagonalized. This fact is true for  any cohomogeneity one metric in all dimension and seems to be new. It may thus be of independent interest for other problems on cohomogeneity one manifolds.  In our case this enables us to show that there exists an interval on the regular part (but not necessarily including the singular orbit) where the metric is diagonal. The resulting equations can  be solved again, although the computations are now significantly more complicated since we cannot reduce the problem by using the smoothness conditions at a singular orbit. Surprisingly, even in this general situation, where we are only looking for local solutions, the problem is very rigid. The solutions either extend smoothly to one of the examples in Theorem A, or they belong to five special families, see Examples 5, 9 and 10 in Section 3. They depend on a parameter $a$ and, if $a$ is an integer, it represents an orbifold solution, one of which is the smooth example.

The paper is organized as follows. In Section 1 we review the geometry of cohomogeneity one manifolds and discuss the condition of being curvature homogeneous. In Section 2 we show when the matrix of functions describing the metric can be partially diagonalized.   In Section 3 we describe the known examples, in terms of the functions describing the metric, of the 4 dimensional cohomogeneity one metrics on symmetric spaces, as well as the example by Tsukada. In Section 4 we discuss the smoothness conditions and in Section 5 solve the case when the metric is diagonal. In Section 6  we discuss the general case.

We would also like to thank the referee for  many helpful suggestions.

%%%%%%%%%%%%%%%%%%%%%%%%%%%%%%
%%%%%%%%%%%%%%%%%%%%%%%%%%%%%%%
\section{Preliminaries}%%%%%%%%%%%
%%%%%%%%%%%%%%%%%%%%%%%%%%%%%%
%%%%%%%%%%%%%%%%%%%%%%%%%%%%%%%

\bigskip

Let $M$ be a Riemannian manifold. Then $M$ is curvature homogeneous if, for any $p,q\in M$, there exists a linear isometry $f: T_pM\to T_q M$ that preserves the curvature tensor, i.e. $f^* R_q=R_p$. This is equivalent (cfr. \cite{TV}) to the existence of a metric linear connection $\D$ such that $D R=0$. Parallel translation with respect to $D$ can then be chosen as the isometry $f$ above. Denote by $\nabla$ the Levi-Civita connection of $M$, and let $A=\nabla-D$. Then $A$ satisfies
\begin{equation}\label{curvhom}
\nabla R=A\cdot R.
\end{equation}
To prove the existence of $D$ one can simply argue as follows. Being curvature homogeneous implies that locally there exists an orthonormal frame field with respect to which the components of the curvature tensor are constant. Declaring this frame to be parallel gives a local solution to \eqref{curvhom}  and one can then use a partition of unity to define a global connection.
Vice versa, the existence of a tensor $A$ that satisfies  \eqref{curvhom} implies that $M$ is curvature homogeneous.

If  $G$ acts by isometries such that $\dim (M/G)=1$, a so called cohomogeneity one manifold, we can average the connection to make it $G$-invariant, and still satisfy \eqref{curvhom}, see \cite{V}. Thus it is sufficient to define the metric and the tensor $A$ only along a geodesic. I.e., if $\gamma(t)$ is an arc length parameterized  geodesic orthogonal to all hypersurface orbits of $G$, then $M$ is curvature homogeneous if and only if there exists a skew-symmetric $(1,1)$-tensor $A(t)$, defined at the regular points of $\gamma$, such that
\begin{eqnarray}\label{diffeq}
	\label{cheq}
	\frac d{dt}\,R(E_i,E_j,E_k,E_l)&=&\sum_m a_i^m R(E_m,E_j,E_k,E_l)+a_j^m R(E_i,E_m,E_k,E_l)+\\ \notag
	&+&a_k^m R(E_i,E_j,E_m,E_l)+a_l^mR(E_i,E_j,E_k,E_m)
\end{eqnarray}
where $A(E_i)=\sum a_i^m E_m$ for a $\nabla$ parallel orthonormal basis $E_i(t)$ of the tangent space at $\gamma(t)$.  Thus if \eqref{diffeq} is satisfied along the geodesic $\gamma$, the action of $G$ guarantees that it is satisfied on all of $M$. Notice also that the tensor $A$ needs to be defined only at the regular points, and we need to check \eqref{diffeq} only at these regular points, since the metric is then curvature homogeneous on all of $M$ by continuity. This avoids the technical issues of having to consider smoothness conditions at the singular point for $A$, in fact $A$ may not even extend continuously to $M$. As we will see, it is also important for us to be able to change the normal geodesic $\gamma$ by an equivariant diffeomorphism in order to simplify the computations.

\smallskip

Before we continue, let us recall the general structure of a cohomogeneity one manifold see, e.g.,  \cite{AA,AB} for a general reference. Since we assume that $M$ is simply connected, it follows that there are no exceptional orbits. The case where all orbits are regular is special and will be solved separately, see Section 6 case (b). So from now on we will assume that there exist some orbits which are singular.
In that case a non-compact cohomogeneity one manifold is given by a homogeneous vector bundle
and a  compact one by  the union of two
homogeneous disc bundles. In our proof it will sufficient to solve the differential equation on a disc bundle near one singular orbit.  The solutions will then determine whether the metrics extends to a complete metric on a vector bundle, or to a compact manifold. We will thus, from now on, also assume that we work simply on a single disc bundle. To describe the disc bundle, let    $H,\, K ,\, G
$ be compact Lie groups with inclusions $H\subset K \subset G$ such that
$K/H=\Sph^{\ell}$ for some $\ell>0$. The transitive action of $K$ on
$\Sph^{\ell}$ extends (up to conjugacy) to a unique linear action on  $V=\R^{{\ell}+1} $, see e.g. \cite{Be}, Theorem 7.50.
We can thus define the homogeneous vector bundle
$M=G\times_{K}V=\{[g,v]\mid (g,v)\sim (gk^{-1},kv) \text{ for any } k\in K \}$. $G$  acts on $M$  via $\bar g [g,v]=[\bar g\cdot g,v]$. A disc $\Disc\subset V$ can be viewed as the slice of the $G$ action since, via the exponential map,  it can be identified $G$-equivariantly with a submanifold of $M$ orthogonal to the singular orbit. Let $p_0=[e,0]$ be a point in the singular orbit $G\cdot p_0=\{[g,0]\mid g\in G\}\simeq G/K$ and $\gamma(t)=[e,t e_0]$ a line in the slice $V$. Then the stabilizer group of $G$ along $\gamma(t)$ is equal to $K$ at $p_0$ and constant equal to $H$ at $\gamma(t)$ for  $t>0$. Under the exponential map the image of $\gamma$ is a geodesic in $M$ orthogonal to all orbits. It is thus sufficient to describe $G$-invariant metrics on $M$ only along  $\gamma(t)$  since $G\cdot\gamma=M$. Conversely, given a cohomogeneity one manifold $M$, the slice theorem implies that the manifold in the neighborhood of a singular orbit has the above form after we choose a normal geodesic $\gamma$ orthogonal to the singular orbit $G/K$ with $\gamma(0)=p_0$.

\smallskip

 We fix a bi-invariant metric $Q$ on $\fg$, which defines a $Q$-orthogonal $\Ad_H$-invariant splitting $\fg=\fh\oplus\fn$. The tangent space $T_{\gamma(t)} (G\cdot \gamma(t))={\gamma'}^\perp\subset T_{\gamma(t)}M$,  is then identified with $\fn$ for $t>0$
via action fields: $X\in\fn\to
X^*(\gamma(t))$. $H$ acts on $\fn$ via the adjoint representation
and a $G$-invariant metric on $G/H$ is described by an $\Ad_H$-invariant inner product on $\fn$. For $t>0$ the metric along $\gamma$
is thus given  by $g=dt^2+h_t$ with $h_t$ a one parameter family of $\Ad_H$-invariant inner products on the vector space $\fn$, depending smoothly on $t$. Conversely, given such a family of inner products $h_t$, we define the metric  on the regular part of $M$ by using the action of $G$. We  describe the metric in terms of a one parameter family of self adjoint endomorphisms:
$$
P_t\colon \fn\to\fn,\quad g(X^*(\gamma(t)),Y^*(\gamma(t)))=Q(P_tX,Y) \text{ for all } X,Y\in\fn.
$$
Since $\Ad_H$ acts by isometries in $g$ and $Q$, $P_t$ commutes with $\Ad_H$.

We choose an $\Ad_H$-invariant  splitting
$$
\fn=\fn_0\oplus\fn_1\oplus\ldots\oplus\fn_r,
$$
where $\Ad_H$ acts trivially on $\fn_0$ and irreducibly on $\fn_i$ for $ i>0$. On  $\fn_i, i>0$, the inner product $h_t$ is  a multiple of $Q$, whereas on $\fn_0$ it is arbitrary. Furthermore, $\fn_i$ and $\fn_j$ are orthogonal if the representations of $\Ad_H$ are inequivalent. If they are equivalent, inner products are described by $1,2$ or $4$ functions, depending on whether the equivalent representations are orthogonal, complex or quaternionic.

Next, we  choose a  basis $X_i$ of $\fn$, adapted to the above decomposition, and thus the metrics $h_t$ are described by a collection of smooth functions $g_{ij}(t)=g(X_i^*(\gamma(t)),X_j^*(\gamma(t)))=Q(P_tX_i,X_j)$, $\ t> 0$.  In order to be able to extend this metric smoothly to the singular orbit, they must  satisfy certain  smoothness conditions at $t=0$, which are discussed in \cite{VZ}.

\smallskip

Choosing an $\Ad_K$-invariant complement to $\fk\subset\fg$, we obtain the $Q$-orthogonal decompositions
$$\fg=\fk\oplus \fm  , \quad  \fk=\fh\oplus \fp \ \text{ and thus }\ \fn=\fp \oplus \fm . $$
where we can also assume that $\fn_i\subset\fp$ or $\fn_i\subset\fm$.   Here $\fm$ can be viewed as the tangent space to the singular orbit $G/K$ at
$p_0=\gamma(0)$ and $\fp$ as the tangent space of the sphere $K/H\subset V$. $K$ acts via the
isotropy action $\Ad(K)_{| \fm}$ of $G/K$ on $\fm$ and via the slice
representation on $V$. The (often ineffective) linear action of $K$ on $V$ is determined by the fact that $H$ is the stabilizer group at $\gamma(t),\ t>0$ and $K/H=\Sph^\ell$. We will also discuss, in Section 6, the case where  
all orbits are regular.

It is important for us to note that the tubular neighborhood $G\times_{K}D$ is not only defined in a neighborhood of the singular orbit, but in fact for all $t$ in the complete non-compact case, and until it reaches the second singular orbit in the compact case.  Once we solve the condition for being curvature homogeneous near the singular orbit,  the metric on $M$ is well defined   since the condition 
involving the tensor $A$ is a regular ODE for $t>0$. In practice, we will recognize the metric in a neighborhood of the singular orbit as a known example, and hence they must agree globally.

\section{Use of the normalizer}
We will not try to solve the system \eqref{cheq} for the most general $G$-invariant metric on $M$ since the conditions are too complicated. Instead we will use the degrees of freedom given by the action of the $G$-equivariant diffeomorphisms on $M$ to show that, if a solution exists, then it is possible to find an isometric solution s.t. the expression of the metric is simpler.

\subsection{Change of the metric on a homogeneous space}
Let $N$ be a Riemannian manifold on which a compact Lie group $G$ acts transitively, almost effectively, and by isometries. If we fix a point $p_0\in N$, and let $H$ be the isotropy subgroup at $p_0$, then $N$ can be identified with $G/H$.  We fix a bi-invariant metric
$Q$ on $G$ and  a $Q$-orthogonal decomposition $\fg=\fh+\fn$ of $\fg$ such that $T_{p_0}N \simeq \fn$ via action fields.
The $G$-invariant  metric $g$ on $N$ is identified with 
an endomorphism  $ P\colon \fn\to\fn$ via $g(X^*(p_0),Y^*(p_0))=Q(P (X),Y)$ for all $ X,Y\in\fn$.

The normalizer $L=N^G(H)/H$ acts, after fixing a base point, on the right, i.e. $R_n(gH)= gn^{-1}H$, or equivalently $R_n(gp_0)=gn^{-1}p_0$ for $n\in L$. $L$ acts freely on $N$ and transitively on the fixed point set $N^H$ and thus $N^H\simeq L$.
The Lie algebra of $N(H)$ is the centralizer of $\fh$ in $\fg$ and hence the Lie algebra of $L$ can be identified with the subalgebra
$\fn_0=\{ X\in \fn \mid Ad(h)X=X  \text{ for all } h \in H \}$ in $\fg$. Notice that $\Ad(g)(\fn_0)\subset\fn_0$ for all $g\in N(H)$ and hence  $L$ acts on $\fn_0$ via the adjoint representation. We can use this action to simplify the metric on $\fn_0$.  

 First notice that for all $p\in N^H$ the $Q$- orthogonal complement $\fn$ of $\fh$ in $\fg$ is independent of $p$ (in fact this is true only for points in $N^H$). Hence we  have the identification via action fields $T_{q}N \simeq \fn$ for all $q\in N^H$ for a {\it fixed} subspace $\fn$.

  We now define a new metric $g'=R_n^*(g)$ for which we have: 
 \begin{align*}
 & Q(P'_t(X),Y)= g'(X^*(p_0),Y^*(p_0))=g\left(d(R_n)_{p_0}(X^*(p_0)),d(R_n)_{p_0}(Y^*(p_0)) \right)  \\
 =& g(X^*(n^{-1}p_0),Y^*(n^{-1}p_0))=g(d(L_{n^{-1}})_{p_0}((\Ad(n)X)^*(p_0)),
 d(L_{n^{-1}})_{p_0}((\Ad(n)Y)^*(p_0)))\\
 = &g((\Ad(n)X)^*(p_0),(\Ad(n)Y)^*(p_0))=Q(P_t\Ad(n)X,
\Ad(n)Y)=Q(\Ad(n^{-1})\,P_t\,\Ad(n)X,Y)
 \end{align*}
 since $ X^*(gp)= d(L_g)_{p} \left((\Ad(g)X)^*(p)\right)$ and 
 $d(R_n)_{p_0}(X^*(p_0)) =X^*(n^{-1}p_0)$.
 
Thus the change in the metric endomorphism is given by
\begin{equation}\label{change}
P'_t=\Ad(n^{-1})\,P_t\,\Ad(n).
\end{equation}
Notice that this can also be interpreted as changing the base point for $g$ from $p_0$ to $n^{-1}p_0$. 

 In particular, if $\Ad(L)_{|n_0}
=SO(\fn_0)$ then we can assume that $P_{|\fn_0}$ is diagonal with respect to a $Q$ orthonormal basis. Notice though that this is possible only if the identity component $L_0$ is isomorphic to  $SO(3)$ or $SU(2)$. Indeed, the action of $L$ on $\fn_0$ is the adjoint action of $L$ on its Lie algebra and if the image is the full orthogonal group,  the maximal torus has dimension $1$ since every vector of unit length can be conjugated into a fixed vector. This implies that the maximal abelian subalgebra of $\fl$ is one dimensional.

\subsection{ Change of the metric on a Cohomogeneity one manifold}
Let $M$ be a cohomogeneity one Riemannian $G$-manifold as in Section 1, with $p_0=\gamma(0)\in G/K$, principal isotropy group $H$, and  $Q$-orthogonal decompositions
$\fg=\fh\oplus\fn=\fh+\fp+\fm$. This  induces the identification $T_{\gamma(t)} G\cdot \gamma(t)\simeq \fp+\fm$  via action fields for all regular points.

We also have the subspaces $\fp_0\subset\fp$ and $\fm_0\subset\fm$ on which $\Ad_H$ acts as $\Id$, and $L=N(H)/H$ acts on $\fp_0\oplus\fm_0$ via its adjoint representation.  As in the previous section we can try to change the metric on $\fp_0\oplus\fm_0$ using the action of $\Ad_L$. We want to understand if we can do this for all points along the normal geodesic smoothly. This will be easy on the regular part, but the smoothness at $t=0$ is more delicate. Notice that for the metric endomorphisms $P_t$ we have that $P_t(\fp_0\oplus\fm_0)\subset \fp_0\oplus\fm_0$ by Schur's Lemma, since $P_t$ commutes with $\Ad_H$.

\begin{prop}\label{normalize}
	Let $(M,G)$ be a  cohomogeneity one manifold with $L=N(H)/H$.
	\begin{enumerate}
		\item[(a)] Assume that $L_0=N(H)/H$ isomorphic to $SO(3)$ or $SU(2)$. If the action has a singular orbit at $t=0$, then there exists an $\e>0$ such that any cohomogeneity one metric is $G$-equivariantly isometric to one where $Q(P_t(\frak{p}_0),\frak{m}_0 )=0$  on  $(0,\e)$.
		\item[(b)] If $L_0=N(H)/H$ isomorphic to $SO(3)$ or $SU(2)$ and there exists an interval $[a,b]$ with $a>0$ on which the eigenvalues of ${P_t}_{|\fn_0}$  have constant multiplicity, then the metric is $G$-equivariantly isometric to one where ${P_t}_{|\fn_0}$ is diagonal  for all $t\in (a,b)$.
		\item[(c)]  Assume that $L_0=N(H)/H$ is isomorphic to $SO(3)$,  $SU(2)$ or $SO(2)$ respectively, and that  there exists a three dimensional, respectively two dimensional $\Ad_H$-invariant  subspace   $\fm_1\subset \fn_0^\perp\cap(\fp\oplus\fm)$ which is also invariant under $\Ad_{N(H)}$.  If there exists an interval $[a,b]$ with $a>0$ on which the eigenvalues of
		${P_t}_{|\fm_1}$ have constant multiplicity, then the metric is $G$-equivariantly isometric to one where ${P_t}_{|\fm_1}$ is diagonal  for all  $t\in (a,b)$.
	\end{enumerate} 
\end{prop}
\begin{proof}(a)
	The assumption implies that $\dim (\fp_0\oplus\fm_0)=3$ and that $\Ad_L$ acts transitively on all $Q$-orthonormal bases in $\fp_0\oplus\fm_0$. We will use this freedom to change the metric. For the transitive actions on spheres we have $\dim\fp_0=0,1 $ or $3$. If $\dim\fp_0=0 $ or $3$ there is nothing to prove. Thus we can assume, from now, on that $\dim\fp_0=1 $.  Notice that this happens precisely when $K$ contains a normal subgroup isomorphic to $U(n)$, $SU(n)$ or $Sp(n)\cdot T^1$  acting linearly on the slice.
	
	Along a normal geodesic the stabilizer group is constant at all regular points and hence, after fixing the choice of a principal isotropy group $H$, we can assume that for all $G$-invariant metrics a normal geodesics lies in $M^H$.

	The normalizer  $L=N(H)/H$ acts on the fixed point set $M^H$ (on the left) and under this action $M^H$ is also a cohomogeneity one manifold with $M/G=M^H/L$.
	On the regular part $M^H_{reg}$ the stabilizer group is constant equal to $H$ and hence we have the identification via action fields $T_q(G\cdot q)\simeq \fp\oplus \fm$   for all $q\in M^H_{reg}$ with respect to a {\it fixed} subspace $ \fp\oplus\fm\subset \fg$.  Under this identification $d({L_g}_{|M_{reg}})_p=\Ad(g)_{|\fp\oplus \fm}$ for all $p\in  M^H_{reg} $ and $g\in N(H)$ since $d(L_g)_{p_0}(X^*(p_0))=(\Ad(g)X)^*(gp)$.
	
	We also have an action of $L$ on the right on $M_{reg}$ (after fixing the geodesic $\gamma$)  via $g\gamma(t)\to gn^{-1}\gamma(t)$ for $n\in L$. It  acts freely on $M_{reg}$ and transitively on each $L$-orbit in the fixed point set $M^H_{reg}$. In general this action will not extend to all of $M$ though, unless $n$ also normalizes $K$.

		We will make use of the  Weyl group element $\sigma\in K$ defined by $\sigma(p_0)=p_0$ and  $d(L_{\sigma})_{p_0}(\gamma'(0))=-\gamma'(0)$, and thus $\sigma(\gamma(t)=\gamma(-t)$. Clearly, this defines $\sigma$ uniquely $\!\!\!\!\mod H$, and $\sigma^2\in H$ as well as $\sigma \in N(H)$. 
	According to the above, we also have $d({L_\sigma}_{|M_{reg}^H})_p=\Ad(\sigma)_{)_{\fp_0\oplus\fm_0}}$, which will be useful for us when $p=\gamma(t)$.

	Since $\Ad(\sigma)$ normalizes $\Ad_H$, it preserves $\fp_0\oplus\fm_0$, and since $\sigma^2\in H$, it follows that  $\Ad(\sigma^2)=\Id$ on $\fp_0+\fm_0$. Hence 	 $\fp_0+\fm_0$ is the sum of two eigenspaces $W_-$ and $W_+$ of $\Ad(\sigma)$ corresponding to the eigenvalues $\pm 1$.

	 The geodesic $\gamma(t)$ and the metric $P_t$ are defined and smooth on an interval around $t=0$ and satisfy
	$$P_{-t}=\Ad(\sigma) P_t\Ad(\sigma)^{-1}$$
	since $\Ad(\sigma)=d(L_\sigma)_{\gamma(t)}\colon T_{\gamma(t)}M\to T_{\gamma(-t)}M$ is an isometry and takes $\gamma(t)$ to $\gamma(-t)$.
	Notice that, although $\sigma$ is only defined mod $H$, this is well defined since $\Ad_H$ commutes with $P_t$.
	
	By assumption $\dim\fp_0=1$ and hence there exists a unique eigenvalue $\lambda_1(t)$ of $P_t$ with $\lim_{t\to 0}\lambda_1(t)=0$. Hence we can choose a $\delta>0$ such that for $t\in(-\delta,\delta)$ the eigenvalue $\lambda_1(t)$ is distinct from all other eigenvalues of $P_t$. We restrict the remaining discussion to this interval only.  Thus
	$\lambda_1(t)$ is a smooth function  (see e.g. \cite{L}, Thm. 8, pg. 130) and there exists a smooth eigenvector $Y_1(t)\in\fp_0$   which we normalize to have unit length in $Q$. It follows that $\Ad(\sigma)Y_1(t)$ is an eigenvector of $P_{-t}$ with eigenvalue $\lambda_1(t)$ and hence there exists an $\e_1=\pm 1$, independent of $t$, such that
	$$Y_1(-t)=\e_1\Ad(\sigma) Y_1(t).$$
	We want to show that this equation holds for a $Q$ orthonormal basis of $\fp_0\oplus\fm_0$, which will imply \eqref{normalizer}, and is crucial in defining the equivariant diffeomorphism $\phi$ later on.
	
	Fix a $Q$-orthonormal basis $X_i\in \fp_0\oplus\fm_0$ such that $\Ad(\sigma)X_i=\e_i X_i$ with $\e_i=\pm 1$ and $X_1=Y_1(0)$. We claim that we can extend $X_i$ to three $Q$-orthonormal vector fields $Y_i(t)$ along $\gamma$ with $Y_i(t)\in \fp_0\oplus\fm_0$ and 
	\begin{equation}\label{Yi}
	Y_i(-t)=\e_i\Ad(\sigma) Y_i(t) \text{ and } Y_i(0)=X_i, \quad i=1,2,3.
	\end{equation}
	We already defined the vector field $Y_1(t)$ satisfying this equation. Next choose a vector field $\bar Y_2(t)\in \fp_0\oplus\fm_0$ orthogonal to $Y_1(t)$ with $\bar Y_2(0)=X_2$ and define
	\begin{eqnarray*}
		Y_2(t)=\frac12(\bar Y_2(t)+\e_2\Ad(\sigma)\bar Y_2(-t))
	\end{eqnarray*}
	Notice this implies that $Y_2$ satisfies \eqref{Yi} with $Y_2(0)=X_2$. We also have
	\begin{align*}
	&2 Q(Y_1(t),Y_2(t))= Q(Y_1(t),\bar Y_2(t)+\e_2\Ad(\sigma)\bar Y_2(-t))=
	 Q(Y_1(t),\e_2\Ad(\sigma)\bar Y_2(-t))\\
	&=  Q(\Ad(\sigma)Y_1(t),\e_2\bar Y_2(-t))= Q(\e_1 Y_1(-t),\e_2\bar Y_2(-t))=0
	\end{align*}
	Similarly, we can find a vector field $Y_3(t)\in \fp_0\oplus\fm_0$ satisfying \eqref{Yi} and orthogonal to $Y_1(t)$ and $Y_2(t)$. Finally, normalize $Y_i(t)$ to have unit length in $Q$.
	
	Since $\Ad_L$ acts  transitively on the set of $Q$-orthonormal basis in $\fp_0\oplus\fm_0$, there exists an element $n_t\in L$ such that $$\Ad(n_t)X_i=Y_i(t) \text{ for } t\in(-\e,\e).$$
	Furthermore, $n_t$ is uniquely determined since the center of $L$ is finite and $n_0=e$.
	Thus $n_t$ depends smoothly on $t$ since $\Ad(n_t)_{\gamma(t)}$ is an isomorphism for all $ t\in(-\e,\e)$.  Finally, we choose $\e'>0$ such that $n_t=e$ for $t>\e+\e'$. Here we point out that $X_i$ should be interpreted as action fields along $\gamma$, whereas $Y_i(t)$ are simply vector fields along $\gamma$.
	
	We next observe that the curve $n_t$ has the crucial property
	\begin{equation}
	\label{normalizer}
	n_{-t}=\sigma\cdot n_t\cdot \sigma^{-1}
	\end{equation}
	modulo elements of $H$, since \eqref{Yi} implies that
	\begin{align*}
	\Ad(n_{-t})\Ad(\sigma)X_i=\e_i\Ad(n_{-t})X_i=\e_i Y_i(-t)=\Ad(\sigma)Y_i(t)=\Ad(\sigma)\Ad(n_{t})X_i
	\end{align*}
	for all $i$.
	
	Using the curve $n_t$ in $L$, we define a map $\phi:M_{reg}\to M_{reg}$ by using the right action of $L$ on each orbit:
	$$\phi(g\gamma(t))= gn_t^{-1}\gamma(t),\quad  t\in(-\e,\e),\ t\ne 0.$$
	This is clearly well defined for a fixed $t>0$ since $\Ad_L$  acts freely on each regular orbit.
	Furthermore, each $G$ orbit intersects the geodesic in precisely two points, namely $\gamma(t)$ and $\sigma\gamma(t)=\gamma(-t)$. To see that $\phi$ is well defined on $M_{reg}$ we use  \eqref{normalizer}:
	$${
	\phi(\sigma\gamma(t))=\sigma n_t^{-1}\gamma(t), {\text{ \it and }} \phi(\gamma(-t)) =n_{-t}^{-1}\gamma(-t)=n_{-t}^{-1}\sigma\gamma(t) =\sigma n_t^{-1}\gamma(t) }.
	$$

	Finally,  $\phi$ can be extended to the singular orbit since $n_0=e$. Altogether this implies that $\phi\colon M\to M$ is well defined, continuous, and $G$-equivariant. It is smooth on $M_{reg}$ since $M_{reg}=I\times G/H$ for some open interval $I$ and $\phi(t,gH)=(t,gn_t^{-1}H)$ which is smooth since  $n_t$ is. To see that it is also smooth at the singular orbit we can argue as follows. Let $U(0)\subset \fm$ be a neighborhood of $0$, with $\bar U=\exp_G(U(0))$, such that $\bar U\to G/K,\ g\to gK$ is a diffeomorphism onto its image. Then, if $\bar V\subset V$ is also a small neighborhood of $0$ in the slice, the map $(g,v)\to gv$ is a diffeomorphism of $\bar U\times \bar V$ onto a small neighborhood of $\gamma(0)$ in $M$. In these coordinates,  $(e,te_0)=\gamma(t)$ and $\phi$ becomes
	$g\cdot\gamma(t)=g(e,te_0)=(g,te_0)\to(gn_t^{-1},te_0) $ and is hence smooth.
	
	We now change the metric $g$ to the new  metric $ \bar g=\phi^*(g)$ for which the curve  $\bar\gamma(t)= n^{-1}_t\gamma(t)$ is a  geodesic normal to all orbits. The metric endomorphism $P_t$ with respect to the geodesic $\bar\gamma$,  changes, according to \eqref{change}, into 
	\begin{equation}
	\bar	P_t=\Ad(n^{-1}_t)\,P_t\,\Ad(n_t)
	\end{equation}
	 and, since $\Ad(n_t)X_i=Y_i(t)$, we get
	$$\bar	P_t(X_1)=\Ad(n_t^{-1})P_t\Ad(n_t)X_1=\Ad(n_t^{-1})P_tY_1(t)=
	\Ad(n_t^{-1})\lambda_1(t)Y_1(t)=\lambda_1(t)X_1.$$
	Since $X_2$ and $X_3$ belong to a different eigenspace of $\bar P_t$, it follows that $Q(\bar{P_t}(\frak{p}_0),\frak{m}_0 )=0$.

	\smallskip 
	
	(b) and (c) The proof in these cases works similarly, in fact is simpler since we are staying away from the singular orbit. Choose $\e>0$ such the eigenvalues  on $(a-\e,b+\e)$ have constant multiplicity and $a-\e>0$. This implies that
	 the eigenvectors are smooth and we choose, for  $t\in (a-\e,b+\e)$,  an orthonormal basis of  eigenvectors of ${P_t}_{|\fp_0\oplus\fm_0}$, respectively ${P_t}_{|\fm_1}$.
	 Let $X_i=Y_i(0)$ considered as action fields along $\gamma$.
	 Choose a smooth curve $n_t\in N(H)/H$ such that 
	$\Ad(n_t)X_i=Y_i(t)$ for   $t\in(a-\e/2,b+\e/2)$ and extend the curve such that $n_t=e$ outside $ (a-\e,b+\e)$. This smooth curve $n_t$ defines a $G$-equivariant diffeomorphism $\phi(g\gamma(t))= gn_t^{-1}\gamma(t)$ and the rest of the proof is as before since $n_t=e$ on an interval around $t=0$.
	\end{proof}

Notice though that in general $\Ad_{N(H)}$ only takes $\Ad_H$-invariant subspaces into other $\Ad_H$-invariant subspaces but does not necessarily preserve them unless they are not equivalent to any other $\Ad_H$-invariant subspace.  Notice also that in part (b) and (c)  in fact no singular orbit is required.
	\smallskip
	
	\begin{rem}
	The proof also shows that one can change the normal geodesic via an equivariant diffeomorphism  into another curve transverse to all orbits.
	Indeed, let $\gamma_i(t),\ i=1,2$  be two smooth curves in $M^H$  transverse to all orbits  satisfying $\gamma_i(-t)=\sigma_i \gamma_i(t)$, where $\sigma_i$ is the Weyl group element at $\gamma_i(0)$. Then there exists a function $f$ with $f(t)=f(-t)$ and a smooth curve $n_t\in N(H)/H$ such that $\gamma_2(t)=n_t^{-1}\gamma_1(f(t))$ at the regular points. The map $\phi(g\gamma_1(t))= gn_t^{-1}\gamma_2(f(t))$ is an equivariant diffeomorphism taking $\gamma_1$ to $\gamma_2$. The condition \eqref{normalizer} becomes $\sigma_2 n_{-t}=n_t \sigma_1$ which follows from $\gamma_i(-t)=\sigma \gamma_i(t)$. The well definedness and smoothnesss of $\phi$ then follows as above. Of course if there are no non-regular orbits there are no conditions.   In particular, in order to describe the set of all $G$-invariant metrics one can, up to an equivariant isometry,  fix a curve and its parametrization, such that it is a geodesic normal to the orbits for all $G$-invariant metrics.
\end{rem}

	 \smallskip

\section{Four dimensional cohomogeneity one manifolds}

In this section we describe all simply connected 4-dimensional cohomogeneity one manifolds  $(M^4,G)$ with $G$  compact, and such that not all orbits are regular, see \cite{GZ,Pa}. Recall that this implies that there are no exceptional orbits. Thus there exists a singular isotropy group, which we denote by $K$.   The principal orbit $G/H$ is a compact  3 dimensional homogeneous space and hence one of $\Sph^3/\Z_k, \ \Sph^2\times\Sph^1$, $\RP^2\times\Sph^1$ or $T^3$.  In the first case, $G=SO(4),U(2),SU(2)$, or $SO(3)$. In the first three sub cases, a normal subgroup $SU(2)$ acts transitively and hence we can assume $G=SU(2)$ since the set of invariant metrics can only get enlarged. If $G=SO(3) $, we can make the action ineffective and again assume $G=SU(2)$. If $G/H= \Sph^2\times\Sph^1$ or $\RP^2\times\Sph^1$, then   $G= SO(3)SO(2)$. This case is special and will be easily dealt with at the end of Section 4. If $G/H=T^3$, then $G=T^3$ and the  principal isotropy group is trivial. Thus $K=S^1$, hence each disk bundle is homotopy equivalent to $T^3/S^1=T^2$, whereas the principal orbit is $T^3$. But then van Kampen implies that $M$ is not simply connected. Thus, except  in subsection 5.4.7, we will from now on  assume that $G=SU(2)$. Hence the principal orbit is $SU(2)/H$ with $H$ a finite group. Such finite groups are isomorphic to $\Z_k$, a binary dihedral group $D_k^*$ (of order $4k$), or the binary groups $T^*,O^*,I^*$. For the three latter groups the isotropy representation of $\Ad_H$ on $\fn$ is irreducible. This means that the only possibility for $K$ is $G$ itself, contradicting the assumption that $K/H$ is a sphere. If $H=\Z_k$, then $K=SO(2)$ or $\Pin(2)$, and if $H=D_k^*$, then $K=\Pin(2)$ are the only possibilities, apart from the special case where $G=K$ and $H=\{e\}$, i.e., where the action has a fixed point.

For the smoothness conditions of the metric near a codimension two singular orbit, i.e. $K_0=S^1$, one needs to consider two integers depending on the action of $K_0$ on $\fm$ and on the slice, see \cite{VZ}. In our case they are both two dimensional and $K_0$ acts with speed 2 on $\fm$ in all cases, and with speed $a=|H\cap K_0|$ on the slice. The integer $a$ depends on the group diagram and will be crucial in our discussion.

\smallskip
We  now describe the isotropy representation, the metrics, and the integer $a$. For convenience we identify $SU(2)$ with $Sp(1)$.

a)\ $K=SO(2)=\{e^{i\theta}\}$ and  $H=Z_k$. The group $H$ is generated by $\zeta=e^{\frac{2\pi i}{k}}$, and hence $\fn=\fn_0\oplus\fn_1\simeq\R\oplus\C$ and $\Ad_H$ acts on $\fn$ as $\Ad(\zeta)(r,z)=(r,\zeta^2 z)$.
If $k=2$, $H$ is effectively trivial and any metric on $\fn$ is allowed. If $k=4$, the action is  $(r,z)\to(r,- z)$ and hence $\fn_0$ is orthogonal to $\fn_1$ and, on the two dimensional module $\fn_1$, the metric can be arbitrary. In all other cases the metric is diagonal and depends on 2 functions.
 $K$ acts on $\fm\simeq \C$ as $v\to A^2v$ since $G/K$ is $\Z_2$ ineffective. Its action on the slice $V\simeq \C$ is given by $v\to A^kv$ since $K/H$ is $\Z_k$ ineffective. Hence $a=k$.
 
 (b) \ $K=\Pin(2)=\{e^{i\theta}\}\cup j\cdot \{e^{i\theta}\}$ and  $H=Z_k$.  In order for $K/H=S^1$, $H$ must contain an element in the second component of $K$, which we can assume is $j\in Sp(1)$. Hence necessarily $k=4$, and $H$ is generated  by $j$ . The action of $H$ on $\fn=\fn_0\oplus\fn_1\simeq\R\oplus\C$ is $(r,z)\to (r,-z)$  hence the metric on $\fn_1$ is arbitrary. Here $a=2$.
 
 (c)\   $K=\Pin(2)=\{e^{i\theta}\}\cup j\cdot \{e^{i\theta}\}$ and  $H=D_k^*, \ k>1$ generated by 
 $\zeta=e^{\frac{2\pi i}{2k}}$ and $j$. If $k=2$, $H$ is the quaternion group. The action of $\Ad_H$ on   $\fn=\fn_0\oplus\fn_1\simeq\R\oplus\C$  is given by $\Ad(\zeta)(r,z)=(r,\zeta^2 z)$ and  $j\cdot(r,z)=(r,b+ci)\to (-r,b-ci)$. The metric is diagonal and depends on $3$ functions. If $k>2$, it depends on only $2$ functions. Here $a=2k$.

\smallskip

(d) \ The last possibility is a codimension 4 singular orbit, i.e. $G=K=SU(2)$, hence $H=\{e\}$ and $a=1$.

\smallskip

A compact manifold is (since we assume not all orbits are regular) the union of two such disk bundles. But in our discussion it will be sufficient to solve the problem near one singular orbit.
\bigskip

We choose the following basis for the Lie algebra of $SU(2)$:
$$X_1=\begin{pmatrix} i &0\\0 &-i\end{pmatrix},\quad X_2=\begin{pmatrix} 0 &1\\-1 &0\end{pmatrix}\quad X_3=\begin{pmatrix} 0 &i\\i &0\end{pmatrix}$$
 with Lie brackets:
$$
[X_i,X_j]=2X_k, \text{ where } i,j,k \text{ is a cyclic permutation of } 1,2,3.
$$
We fix a bi-invariant metric $Q$ on $\fg$ such that $X_i$ is orthonormal. The metric with be determined by the inner products of these basis elements.

\bigskip
\begin{center}
	Known examples
\end{center}
\bigskip

We need to compare our solutions with the known solutions on simply connected homogeneous 4 dimensional manifolds with $G=SU(2)$, considered as a cohomogeneity one manifold. It turns out that these are all symmetric spaces.  For the compact examples the metric is defined on $0\le t\le t_0$ with stabilizer groups $K_-$ and $K_+$ at $t=0$ and $t=t_0$. In the case of non-compact manifolds there is only one singular isotropy group, which we denote by $K$.  An upper index $'$ indicates a different embedding. See \cite{Zi} for some of the details. Notice though that, since we choose $G=SU(2)$ instead of $G=SO(3)$ as in \cite{Zi}, the functions must be multiplied by 2. Since all metrics are diagonal, we list the norm of the elements $X_i$ of the basis, i.e. $f_i=v_i^2$. We separate the examples by the integer $a=|H\cap K_o|$ determined by the smoothness conditions.

\bigskip
\begin{center}
	Ineffective kernel $a=4$
\end{center}

\bigskip

{\bf Example 1} $M=\Sph^4 $ with $G=SU(2)$, $K_-=\Pin(2)$, $K_+=\Pin(2)'$ and $H=D_2^*$, the quaternion group, and metric
\begin{equation*}
v_1 = 4 \sin(t),\  v_2 = 2\sqrt{3} \cos(t) - 2\sin(t),\  v_3 = 2\sqrt{3} \cos(t) + 2\sin(t))
 \text{ with } 0\le t\le \pi/3
\end{equation*}

\bigskip

{\bf Example 2} $M=\CP^2 $ with $G=SU(2)$, $K_-=SO(2)'$, $K_+=\Pin(2)$ and $H=\Z_4$ and metric
\begin{equation*}
v_1 = 2\sin(2t),\ v_2 = \sqrt{2}(\cos(t)-\sin(t)),\  v_3 = \sqrt{2}(\cos(t)+\sin(t))  \text{ with } 0\le t\le \pi/4
\end{equation*}

\smallskip

{\bf Example 3} The Tsukada Example with  $G=SU(2)$, $K=\Pin(2)$ and $H=D_2^*$ and metric 
$$
v_1=2(e^t-e^{-t}),\quad
v_2=2e^t,\quad
v_3=2e^{-t}$$

\bigskip

\begin{center}
	Ineffective kernel $a=2$
\end{center}

\bigskip

{\bf Example 4} $M=\CP^2 $ with $G=SU(2)$, $K_-=\Pin(2)$, $K_+=SO(2)'$ and $H=\Z_4$ and metric
\begin{equation*}
v_1 = 2\sin(t),\quad v_2 = 2\cos(2t),\quad v_3 = 2\cos(t) \quad \text{ with } 0\le t\le \pi/4
\end{equation*}
\smallskip

{\bf Example 5} $M=\Sph^2\times\Sph^2 $ with $G=SU(2)$, $K_-=K_+=SO(2)$ and $H=\Z_2$ and metric
\begin{equation*}
v_1=2\sin( t),\quad  v_2=2\cos(t),\quad  v_3=2,\quad \text{ with } 0\le t\le \pi/2
\end{equation*}
We can also multiply all functions  by a constant $b$. This is an orbifold metric if $b$ is an integer different from $2$. But we remark that if $b<2$, there are some sectional curvatures which are negative.

\bigskip
\begin{center}
	Ineffective kernel $a=1$ (actions with fixed points)
\end{center}

\bigskip

{\bf Example 6} $M=\CP^2$ with $G=SU(2)$, $K_-=SU(2)$, $K_+=SO(2)$ and $H=\{e\}$ and metric
\begin{equation*}
v_1=v_2= \sin(t), \quad v_3= \frac 12 \sin(2t),\quad \text{ with } 0\le t\le \pi/2
\end{equation*}
\smallskip

\smallskip

{\bf Example 7} $M=\C\QH^2$ with $G=K=SU(2)$ and $H=\{e\}$ and metric
\begin{equation*}
v_1 =v_2= \sinh(t),\quad v_3 =\frac 12 \sinh(2t) \quad \text{ with } 0\le t <\infty
\end{equation*}
\smallskip

{\bf Example 8} $G=SU(2)$ and $H=\{e\}$ and metric
\begin{equation*}
v_1 =v_2= v_3=\sin(t),\ \sinh(t), \text{ or } t ,  \quad \text{ with } 0\le t\le \pi
\end{equation*}
\smallskip
corresponding to the standard metric on $\Sph^4$, $\QH^4$ and $\R^4$. In the first case $K_-=K_+=SU(2)$ and in the last two cases $K=SU(2)$.

\bigskip

\begin{center}
$G=SO(3)SO(2)$ with $a=1$ 
\end{center}

\bigskip

{\bf Example 9} $M=\Sph^4$  with $G=SO(3)SO(2)$, $K_-=T^2$, $K_+=SO(3)$,  and $H=S^1\times\{e\}$ and metric
\begin{equation*}
v_1=\sin(t),\quad v_2=v_3=\cos(t) \quad \text{ with } 0\le t\le \pi/2
\end{equation*}
We can also multiply $v_1$  by a constant $a$. This is an orbifold metric if $a$ is an integer with $a>1$.

\bigskip

{\bf Example 10}  $G=SO(3)SO(2)$ and $H=S^1\times\{e\}$ with metric
\begin{equation*}
v_1=\sin(t),\ \sinh(t), \text{ or } t ,\ \text{ and }  v_2=v_3=a 
\end{equation*}
corresponding to the product metric on  $M=\Sph^2\times\Sph^2$,  $M=\Sph^2\times\R^2$ and  $M=\Sph^2\times\QH^2$.  In the first case $K_-=K_+=T^2$ and in the last two cases $K=T^2$.  We can again multiply $v_1$ by a constant, and get an orbifold metric.
\smallskip

\smallskip

Example 2 and 4 are the same action and metric, but the two singular orbits are interchanged.  The functions look different though, which will be useful when solving the equations in Section 5. For example 1 and 2, see \cite{Zi} Section 2 for details.  Example 5 is the diagonal action of $SO(3)$ on $\Sph^2(a)\times\Sph^2(b)$. Example 9 is a sum action on $\R^3\oplus\R^2$ and in Example 10 $SO(3)$ acts transitively on the first factor.

\smallskip

In \cite{Ts} Tsukada uses $G=SO(3)$ instead of $SU(2)$ which means the functions in Example 3 need to be divided by $2$ when comparing.
The metric is clearly complete, and Tsukada proved that it is  not homogeneous by showing that $|\nabla \Ric|$ is not constant. The underlying non-compact manifold can be regarded as one of the homogeneous disc bundles in the Example 1 on $\Sph^4$.  All other examples are isometric to symmetric spaces.

This list of Examples is also a description of all cohomogeneity one actions on simply connected 4-manifolds with $G=SU(2)$ and not all orbits regular, as long as we add the action where  $G=SU(2)$,  $K_-=K_+=SO(2),\ H=Z_n$ and $a=1$. This manifold is $\CP^2 \# -\CP^2$ when $n$ is even, and $\Sph^2\times\Sph^2$  when $n$ is odd.  But, according to our result, it does not admit a curvature homogeneous metric.

\section{The form of the metric}
In this section we discuss the metric and the smoothness conditions. The metric is completely describe by the restriction to the space tangent to the regular orbits along the normal geodesic, and has the form

$$f_i(t)=g_t(X_i^*,X_i^*)_{\gamma(t)},\qquad d_{ij}(t)=g_t(X_i^*,X_j^*)_{\gamma(t)}.$$
for $i,j=1,2,3$ and $i\neq j$. We identify $g_t$ with the matrix
$$P_t=\begin{pmatrix}f_1 &d_{12} &d_{13}\\d_{12} &f_2 &d_{23}\\d_{13} &d_{23} &f_3
\end{pmatrix}.$$

If the singular orbit has codimension two, we have:
$$\fg=\fp_0+\fm$$
where $\dim\fp_0=1$, $\dim\fm=2$ and possibly  $\fm=\fm_0$.

For the smoothness conditions two integers $d=2$ and $a$ (see the notaion in \cite{VZ}) enter. According to \cite{VZ} the metric $P_t$ is smooth at $t=0$ if and only if:
\begin{eqnarray}\label{smooth}
f_1=a^2\,t^2+t^4 \phi_1(t^2),\quad f_2+f_3=\phi_2(t^2),\quad f_2-f_3=t^{\frac 4a}\,\phi_3(t^2)\\
d_{12}=t^{2+\frac 2a}\,\phi_4(t^2),\quad d_{13}=t^{2+\frac 2a}\,\phi_5(t^2),\qquad d_{23}=t^{\frac 4a}\,\phi_6(t^2)\nonumber
\end{eqnarray}
for some smooth functions $\phi_i$.

It follows that the metric is diagonal and $f_2=f_3$ unless $a=1,2,4$ and $d_{12}=d_{13}=0$ unless $a=1,2$. For $a\leq 4$ we have the following possibilities for $K$ and $H$ according to Section 3:

\begin{enumerate}
	\item If $a=1$, then $K=SO(2)$ and $H=\{e\}$.
	\item If $a=2$, then $K=\Pin(2)$ and $H=\Z_4$, or $K=SO(2)$ and $H=\Z_2$.
	\item If $a=4$, then $K=\Pin(2)$ and $H=D_2^*$,  or $K=SO(2)$ and $H=\Z_4$.
\end{enumerate}

\smallskip

If the singular orbit has codimension $4$, smoothness according to \cite{VZ} means that
$$
f_{i}=t^2 +t^4\phi_{i}(t^2),\qquad d_{ij}=t^4\phi_{ij}(t^2)
$$
for some smooth functions $\phi_i,\phi_{ij}$.

The last case, where the action has no singular orbits, no smoothness conditions are required, apart from the functions $f_i, d_{ij}$ being smooth for all $t$. 

We finally remark that if the manifold is non-compact, the metric is complete if it is defined for all $t\in \R$.

\section{Diagonal metrics}

In this Section we consider smooth diagonal metrics, i.e., metrics where $d_{ij}=0$, on simply connected cohomogeneity one manifolds with at least one singular orbit. The case where the action has a fixed point  is easy and left to the end of the section and the case where all orbits are regular will be dealt with separately in Section 6. So we assume that $G=SU(2)$ and  $K_0=SO(2)$. Here it will be simpler to express the functions as length $f_i=v_i^2$. Then
the vector fields $E_i=\frac 1{v_i}\,X_i$, for $i=1,2,3$, and $E_4=\gamma'(t)$ form an orthonormal basis of $T_{\gamma(t)}M$ for all $t\ge 0$. In this case the only non-zero components of the curvature tensor are (see \cite{GVZ}):
\begin{eqnarray}\label{curvature}
	R(E_i,E_j,E_i,E_j)&=&\frac{2v_k^2(v_i^2 + v_j^2) - 3v_k^4 + (v_i^2 - v_j^2)^2}{v_i^2v_j^2v_k^2}
	-\frac{v_i'}{v_i}\frac{v_j'}{v_j} \nonumber\\ 
	R(E_i,E_j,E_k,E_4)&=&\sigma(-2\frac{ v_k'}{v_iv_j} +
	\frac{v_i'}{v_i}\frac{v_i^2 + v_k^2 - v_j^2}{v_iv_jv_k} + \frac{v_j'}{v_j}\frac{v_j^2 +
		v_k^2 - v_i^2}{v_iv_jv_k})\\
	R(E_i,E_4,E_i,E_4)&=&-\frac{v_i''}{v_i} \nonumber
\end{eqnarray}
where $\sigma$ is the sign of the permutation $(i,j,k)$.

In order to apply \eqref{diffeq}, we make the following general observation:
\begin{lem}\label{parallel}
	Let $M$ be a Riemannian cohomogeneity one manifold, and $X_i^*$ Killing vectorfields which form a basis of $(\gamma')^\perp$ along a geodesic $\gamma$ normal to all orbits. If $X_i^*$ are orthogonal to each other along $\gamma$, then the normalized vector fields $\bar X_i=X_i^*/|X_i^*|$ are parallel along $\gamma$.
\end{lem}
	\begin{proof}
	Since $\ml  X_i^*, X_j^*\mr=0$ for $i\ne j$, it follows, using $\ml[X_i^*,X_j^*],\gamma'\mr=0$, that 
	\begin{eqnarray*}
	0=\ml  X_i^*, X_j^*\mr'=\ml \nabla_{\gamma'}  X_i^*, X_j^*\mr
	+\ml   X_i^*, \nabla_{\gamma'} X_j^*\mr=-\ml \nabla_{X_j^*}  X_i^*, \gamma'\mr
	-\ml  \nabla_{X_i^*}X_j^*,\gamma' \mr \\
	=-2	\ml \nabla_{X_j^*}  X_i^*, \gamma'\mr=2	\ml \nabla_{ \gamma'}  X_i^*,X_j^*\mr=2 |X_i^*|	\ml \nabla_{ \gamma'} \bar X_i,X_j^*\mr.
\end{eqnarray*}

Furthermore, $\ml \nabla_{ \gamma'} \bar X_i,\bar X_i\mr =\frac12  \ml  \bar X_i,\bar X_i\mr'=0 $, and  $\ml \nabla_{ \gamma'} \bar X_i,\gamma'\mr=\frac{1}{|X_i^*|}\ml \nabla_{ \gamma'}  X_i^*,\gamma'\mr =0$. Altogether, we have $\nabla_{ \gamma'} \bar X_i=0$.
	\end{proof}
Since action fields are Killing vector fields along a geodesic, it follows that $E_i$ are parallel along $\gamma$.

From \eqref{curvature} it follows that if $R(E_i,E_j,E_k,E_l)\ne 0$, then  $R(E_m,E_j,E_k,E_l)\ne 0$  if and only if $m = i$. Together with the skew symmetry of $A$, it follows that the right hand side in  \eqref{diffeq} is $0$.
Hence, if $M$ is curvature homogeneous,  all components of the curvature tensor  are constant (for a diagonal metric). We will use this property and take a  power series at the singular orbit to classify the metrics. Recall also that we can assume that $a=1,2$, or $a=4$.
\smallskip

We start with the case where the singular orbit has codimension $2$.

\smallskip

At the singular orbit $v_1(t)$ satisfies $v_1(0)=0$. The equation   $R(E_1,E_4,E_1,E_4)=-\frac{v_1''}{v_1}=c$ can be integrated explicitly, and we can normalize the metric so that $c=1,-1$ or $0$. Thus $v_1$ must be one of $a\sin(t),\ a\sinh(t)$ or $at$
for some constant  $a$, determined by the ineffective kernel of the action.  We can also solve the equations $R(E_i,E_4,E_i,E_4)=-\frac{v_i''}{v_i}=c_i,\ i=2,3$ for some constants $c_i$. Notice that when $v_1$ is a trigonometric  function the manifold is compact. Since $v_1$ vanishes at $k\pi, k\in\Z$, the stabilizer group is equal to $SO(2)$ at these points and hence  $v_2$ and $v_3$ must agree at $k\pi$ since the metric is invariant under $SO(2)$ at these points.    Thus $v_2$ and $v_3$ are trigonometric functions as well. Similarly, if $v_1$ is hyperbolic or linear, $M$ is non-compact and hence  $v_2$ and $v_3$ are both hyperbolic or linear.

 Hence the metric is determined by a few constants. We can plug the functions $v_i$ in the expressions for the components of the curvature tensor, that have to be constant. By computing the first non-zero derivatives  of the curvatures at $t=0$ we derive some algebraic conditions on these constants which can be solved explicitly.  
 
 \smallskip
 
  It is convenient to first deal with the special case of $v_2(t)=v_3(t)$.

\smallskip

\subsection{$v_2(t)=v_3(t)$}
Notice that smoothness implies that $v_2$ must be an even function. We also have that $R(E_2,E_4,E_2,E_4)=R(E_3,E_4,E_3,E_4)=c$.  
If  $M$ is compact, the smoothness conditions imply that 
$$
v_1=a\sin(t),\quad  v_2=b\cos(ct)\quad v_3=b\cos(ct)$$
with $a,b>0 ,c\ge 0$. 
Computing the first non-zero derivatives of $R(E_1,E_2,E_1,E_2)$ and $R(E_2,E_3,E_2,E_3)$ at $t=0$ gives us two equations:
$$
b^4c^4-4b^2c^2+3a^2=0,\quad b^4c^4-b^4c^2+3a^2=0
$$
which implies $b^2c^2(b^2-4)=0$. Thus $b=2$, and hence $16c^4-16c^2+3a^2=0$.  Using the fact that only $a=1,2,4$ are allowed, it follows  that $a=1$, and  $c=\frac12$.
Thus the only solution is 
$$
v_1=\sin(t),\quad  v_2=v_3=2\cos(\frac12 t)
$$
which is  Example 6 on $\CP^2$. Notice though that we have to switch the two sides and reparametrize the geodesic. 

If $M$ is non-compact, we have one possibility:
$$
v_1=a\sinh(t),\quad  v_2=b\cosh(ct)\quad v_3=b\cosh(ct)$$
with $a,b>0,$ $c\ge 0$.  This gives rise to the equations: 
$b^4c^4+4b^2c^2+3a^2=0,\
b^4c^4-b^4c^2+3a^2=0
$,  thus $b^cc^2(b^2+4)=0$,
which forces $a=0$ and we have no solutions.

The last possibility is that
$$
v_1=a t,\quad  v_2=b\cosh(ct)\quad v_3=b\cosh(ct)$$
with $a,b>0,$ $c\ge 0$.  This gives rise to the equations:
$b^4c^4+4b^2c^2+3a^2=0,\
b^4c^4+3a^2=0
$
 which  has only $0$ solutions.
 
 \bigskip

From now on we may assume that $v_2\neq v_3$. The smoothness conditions at the singular orbit tells us that
\begin{equation*}
\label{smth2}
v_2^2+v_3^2=\phi_1(t^2),\qquad v_2^2-v_3^2=t^{\frac 4a} \phi_2(t^2)
\end{equation*}
for some smooth functions $\phi_1,\phi_2$. If $\frac 4a$ is not an integer this implies $v_2=v_3$, hence we may assume that $a=1,2,4$. But for $a=1$ we have that $v_2$ and $v_3$ must agree at $t=0$ up to order $3$. Since both functions satisfy the ODE $R(E_i,E_4,E_i,E_4)=c_i$, it again follows that $v_2=v_3$. Hence we are just left with the cases $a=2$ and $a=4$. We consider the cases $c=1,0,-1$ separately. Notice also that the smoothness conditions imply that if $v_2$ is hyperbolic and  $v_3$ linear, then $v_3$ must be constant.

\smallskip

\subsection{$c=1$}
Now $v_1=a\sin(t)$ with $a=2$ or $a=4$.
\subsubsection{a=2}
The smoothness conditions imply that $v_2$ and $v_3$ are even with $v_2(0)=v_3(0)\neq 0$ and hence
$$v_1=2 \sin(t),\quad v_2=b \cos(c_1 t), \quad v_3=b \cos(c_2 t)
$$
for some constants $b>0, c_i\ge 0$. From the first non-zero derivatives of $	R(E_1,E_2,E_1,E_2),$ $ R(E_1,E_3,E_1,E_3),$ $ R(E_2,E_3,E_2,E_3)$
we get the equations:
$$
A=b^4c_1^4-6b^4c_1^2c_2^2+b^4c_2^4+8b^2c_1^2+8b^2c_2^2-48=0
$$
$$
B=b^4c_1^4-6b^4c_1^2c2^2+b^4c_2^4-
4b^4c_1^2+8b^4c_2^2+24b^2c_1^2-24b^2c_2^2-48=0
$$
$$
C=b^4c1^4-6b^4c_1^2c_2^2+b^4c_2^4+8b^4c_1^2-4b^4c_2^2-
24b^2c_1^2+24b^2c_2^2-48=0
$$

Hence $\frac43 B-\frac13 C-A=8b^2(c_1^2-\frac32 c_2^2)(4-b^2)=0,\ B-C=12b^2(c_1^2- c_2^2)(4-b^2)=0$, and thus $b=2$ (in which case the 3 equations are identical). Using in addition the derivatives of the remaining components of the curvature tensor one obtains:
$$
(c_1^2-c_2^2)^2-2(c_1^2+c_2^2)+1=0,\quad \text{and}\quad c_1^2+c_2^2-c_1^2c_2^2-1=0
$$
 whose solutions are, up to permutation, $(c_1,c_2)=(1,2),(1,0)$.
The two solutions are  Example 4 on $\CP^2$ and Example 5 on $\Sph^2\times\Sph^2$.

\smallskip

\subsubsection{a=4}
As in the previous case $v_2$ and $v_3$ must be trigonometric functions. From \eqref{smth2} it follows that $v_2+v_3$ is even while $v_2-v_3$ is odd:
$$
v_1=4 \sin(t),\quad
v_2=b_1\cos(c t) + b_2 \sin(c t),\quad
v_3=b_1 \cos(c t) - b_2 \sin(c t)
$$
with $b_1,b_2,c>0$.

If we impose the vanishing of the first few derivatives of the three sectional curvatures at $t=0$,  we obtain the following equations:
$$
c^4(3b_1^4-2b_1^2b_2^2-9b_2^4)-c^2(b_1^2b_2^2+12b_1^2+36b_2^2)+144=0
$$
$$
c^4(b_1^4+2b_1^2b_2^2-3b_2^4) - c^2(b_1^4+2b_1^2b_2^2)+48=0,\qquad c^2(b_1^2-3b_2^2)-b_1^2+12=0 
$$

The last equation implies that that there are two possibilities: $b_1^2=12$ and $b_2^2=4$, or $c^2=(b_1^2-12)/(b_1^2-3b_2^2)$. In the former case the first two equations reduce to $4c^4-7c^2+3=0,\ 4c^4-5c^2+1=0$ and hence $c=1$. In the latter case, substituting $c^2$ into the second  equation gives $3b_1^4(b_2^2 - 4)(b_2^2 + b_1^2 - 16)=0$. If $b_2^2=4$ then  the first equation becomes $b_1^4 - 8b_1^2 - 48=0$ and hence $b_1^2=12$ and thus $c=1$ again. If $b_2^2=16-b_1^2$, then the first equation gives you $b_1^2=8$. Thus $b_2^2=8$ and hence  $c=1/2$.

Thus the only  solutions are:
$$
(b_1,b_2,c)=(2\sqrt{3},2,1) \text{ and } (b_1,b_2,c)=( 2\sqrt{2}, 2\sqrt{2},\frac12)
$$
and we obtain  Example 1 on $\Sph^4$, 
and Example 2 on $\CP^2$ (up to parametrization).

\smallskip

\subsection{$c=0$}
In this case $v_1=at$ with $a=2$ or $a=4$.
\subsubsection{a=2}
As in the previous case, the functions have the form
$$
v_1=2t,\quad v_2=b \cosh(c_1 t), \quad v_3=b \cosh(c_2 t)
$$
with  $b>0$, $c_i\ge 0$ and $c_1\ne c_2$ which satisfy the equations
$$
b^4c_1^4-6b^4c_1^2c_2^2+b^4c_2^4-8b^2c_1^2-8b^2c_2^2-48=0 
$$
$$
 b^4c_1^4-6b^4c_1^2c_2^2+b^4c_2^4-24b^2c_1^2+24b^2c_2^2-48=0 
  $$
 $$ 
  b^4c_1^4-6b^4c_1^2c_2^2+b^4c_2^4+24b^2c_1^2-24b^2c_2^2-48=0
$$
This implies $b^2(2c_1^2-c_2^2)=0$ and $b^2(c_1^2-c_2^2)=0$ and hence there are no solutions.

\smallskip

\subsubsection{a=4}
$$
v_1= 4t,\quad
v_2=b_1\cosh(c t) + b_2 \sinh(c t),\quad
v_3=b_1 \cosh(c t) - b_2 \sinh(c t)
$$
with $b_1,b_2>0, c\ge 0.$ Here we get

$$
c^4(b_1^4-2b_1^2b_2^2-3b_2^4)+48=0  , \quad c^2(b_1^2 + 3b_2^2) - 12=0
$$ 
Solving the second equation for $c^4$ and substituting into the first implies $b_1=0$, contradicting $v_2(0)\ne 0$.

\smallskip

\subsection{$c=-1$}
In this case $v_1=a\sinh(t)$ with $a=2$ or $a=4$.
\subsubsection{a=2}
The functions must have the form 
$$
v_1=2\sinh(t),\quad v_2=b\cosh(c_1t),\quad v_3=b\cosh(c_2t)
$$
with $b>0$, $c_1,c_2\ge 0$ and $c_1\ne c_2$. The  vanishing of the derivatives of the curvature gives:
$$
 b^4c_1^4-6b^4c_1^2c_2^2+b^4c_2^4-8b^2c_1^2-8b^2c_2^2-48=0
 $$ 
$$
 b^4c_1^4-6b^4c_1^2c_2^2+b^4c_2^4-4b^4c_1^2+8b^4c_2^2-  24b^2c_1^2+24b^2c_2^2-48=0
 $$ 
$$
b^4c_1^4-6b^4c_1^2c_2^2+b^4c_2^4+8b^4c_1^2-  4b^4c_2^2+24b^2c_1^2-24b^2c_2^2-48=0
$$
which easily implies that
$$
b^2(c_1^2-c_2^2)(b^2+4)=0, \quad b^2(c_2^2-2c_1^2)(b^2+4)=0,
$$
which has no non-zero solutions.

\smallskip

\subsubsection{a=4} Here the functions are
$$
v_1=4 \sinh(t),\quad
v_2=b_1\cosh(c t) + b_2 \sinh(c t),\quad
v_3=b_1 \cosh(c t) - b_2 \sinh(c t)
$$
with $b_1,b_2,c>0$. These functions need to satisfy the equations
$$ c^4(3b_1^4+2b_1^2b_2^2-9b_2^4)+c^2(b_1^2b_2^2+12b_1^2-   36b_2^2)+144=0
$$
$$ c^2(b_1^2+3b_2^2)-b_1^2-12=0,\quad  c^4(8b_1^2b_2^2 - 12b_2^4) + c^2(b_1^2b_2^2 - 48b_1^2 + 48b_2^2) + 12b_1^2=0
$$

Thus $c^2=(b_1^2+12)/(b_1^2+3b_2^2)$ and substituting into the third equation gives you
$$\frac{9 b_1^4 (b_2^2-4) (-b_2^2+b_1^2+16)}{(b_1^2+3 b_2^2)^2}=0.$$
If $b_2=2$, then $c=1$ and the first equation becomes $b_1^4 + 8b_1^2 - 48$ which implies $b_1^2=4$. If $b_1^2=b_2^2-16$, then the first equation gives you $b_2^2=8$ contradiciting $b_1^2>0$.

Thus we have only one solution, $(b_1,b_2,c)=(2,2,1)$,  and hence the metric is
$$
v_1=4 \sinh(t)=2(e^t-e^{-t}),\quad
v_2=2\cosh( t) + 2 \sinh( t)=2e^t,\quad
v_3=2 \cosh( t) - 2 \sinh( t)=2e^{-t}
$$
which is the Tsukada's Example 3. 

\bigskip

\subsubsection{Codimension 4}

We will finally discuss  the case of a codimension 4 singular orbit with the metric still being diagonal.
Here smoothness implies that $v_i(0)=0$ and $v_i'(0)=1$, $\ i=1,2,3$.
Thus each function is one of 
$$\frac{1}{b}\sin(bt), \quad \frac{1}{b}\sinh(bt), \quad t.$$
This gives rise to the following cases.
\subsubsection{Case 1}
If the manifold is compact, we have the following possibility:
$$v_i=\frac{1}{b_i}\sin(b_it), \quad i=1,2,3$$
which gives rise to the equations
$$
5(b_1^2 - b_2^2)b_3^2 - b_1^4 + b_2^4=0,\quad 5(b_3^2 - b_1^2)b_2^2 + b_1^4 - b_3^4=0
$$
One easily sees that the only solutions are
  $(b_1,b_2,b_3)=(b,b,b)$ and $(b_1,b_2,b_3)=(2b,b,b)$ giving rise to Examples 6 and 8, up to parametrization.

\subsubsection{Case 2}
In the second case, we have
$$v_i=\frac{1}{b_i}\sinh(b_it), \quad i=1,2,3$$
The equations and solutions are the same as in the previous case, giving rise to  Examples 7 and  8.

\smallskip

\subsubsection{Case 3}
Combinations of hyperbolic and linear functions have no solutions. For example
$$v_i=\frac{1}{\sqrt{b_i}}\sinh(\sqrt{b_i}t),\quad i=1,2, \quad v_3=t$$
gives rise to the equations
$$
b_1^2  - 5 b_2^2=0 ,\quad                          
5b_1^2  -b_2^2=0,\quad 3b_2^4 - 3b_1^4=0 
$$
with only $0$ solutions.
\smallskip

The last case is $v_i=t$ which is of course flat euclidean space, again  one of the cases in Example 8.

\bigskip

\subsubsection{G=SO(3)SO(2)}
Here we can assume  that  $H_0=SO(2)\times\{e\}$. Indeed, if we choose a diagonal embedding of $H_0$, it would follow that the subaction by $SO(3)$ is also cohomogeneity one, but this is a special case of the previous discussions. Let $X_1,X_2\in\fn\subset\fso(3)\subset\fg$ be generators such that $Q$ induces the metric of curvature $1$ on $SO(3)/SO(2)=\Sph^2$ or $SO(3)/O(2)=\RP^2$. Finally, let $T$ be a generator of the Lie algebra of $SO(2)$. The decomposition of $\fn$ under the isotropy representation of $H$ is $\fn=\fn_0\oplus\fn_1$. $H_0$ acts trivially on $\fn_0$, and effectively on $\fn_1$ and hence $a=1$. The metric depends only on two functions, $f(t)=|X_1^*|_{|\gamma(t)}=|X_2^*|_{|\gamma(t)}$ and $g(t)=|T^*|_{|\gamma(t)}$. The principal orbit is $\Sph^2\times \Sph^1$, and $f(t)$ represents the radius of $\Sph^2$, and $g(t)$  the radius of $\Sph^1$.  The curvature operator of the metric on $M^4$ is diagonal and hence determined by  the sectional curvatures of the 2-planes spanned by  basis vectors. As before, \eqref{diffeq} implies that the metric is curvature homogeneous if and only if these curvatures are constant. Since
$$
\sec(X_i,\gamma')=-\frac{f''}{f},\quad \sec(T,\gamma')=-\frac{g''}{g}
$$
the functions $f,g$ are again trigonometric, hyperbolic or linear. From the Gauss equations it follows that
$$
\sec(X_1,X_2)=\frac{1-(f')^2}{f^2},\quad 
\sec(X_i,T)=-\frac{f'g'}{fg}
$$
The first equation, since the value must be constant, implies  that $f$ is either a trigonometric function, or $f$ is constant.
In the first case, we can shift the geodesic so that $f=\sin(t)$ and hence  the singular orbit has codimension three. The second equation then implies that $g=a\cos(t)$ for some $a>0$. This means that a second singular orbit occurs at $t=\pi/2$, and smoothness at that singular orbit implies $a=1$. This is Example 9, with the geodesic reversed.  

If $f$ is constant, then  $g$ must vanish at a singular orbit of codimension two, and we can assume, by starting at that singular orbit, that $g=\sin t$, $g=\sinh t$ or $g=t$. This is Example 10.

Notice though that 
$$
v_1=v_2=\sin(t),\quad v_3= a\cos(t)
$$
or, if we reverse the geodesic,
$$
v_1= a\sin(t),\quad v_2=v_3= \cos(t)
$$
is also a  curvature homogeneous solution. It is only smooth if $a=1$, but if $a$ is an integer it is an orbifold solution on $\Sph^4$ with an orbifold angle of $2\pi/n,\ n\in \Z^+$ at the codimension two singular orbit. Similary, in the second example we can multiply $g$ by a constants and get an orbifold metric on $\Sph^2\times\Sph^2$.

\bigskip

\section{The general case.}

Recall that if the metric is not necessarily diagonal, then we have two cases. Either $H$ acts trivially on $\fn=\fh^\perp$ and $N(H)=G=SU(2)$, or $\fn=\R\oplus\C$ with $H$ acting trivially on $\R$ and as $-\Id$ on $\C$ and hence $N(H)/H=S^1$. In the first case the metric is arbitrary on $\fn$, and in the second case $\R$ and $\C$ are orthogonal, but the metric on $\C$ is arbitrary. In either case, there exists an interval $[a,b]$ on the regular part of $M$ such that the multiplicities of $P_t$ are constant.   According to \pref{normalize},  the metric is thus equivariantly isometric to one where it  is diagonal in $I=(a,b)$. We will assume from now on that this is the case. Let $P_t=(v_1,v_2,v_3)$ be this metric with $v_i=|X_i|$. The curvature tensor is again given by the formulas in \eqref{curvature}. Thus it follows that if the metric is curvature homogeneous in the interval $I$, then the three functions $v_i$ are of the form
$$
v_i= a_i \sin(d_i t)+b_i \cos(d_i t), \quad v_i=a_i\, e^{d_i t}+b_i\, e^{-d_i t}, \quad  v_i=a_i\, t+b_i, \quad t\in I
$$
These functions  are analytic on all of $\R$ and the curvature functions, according to \eqref{curvature}, are analytic as long as $v_i>0$.

  Let $J$ be a connected interval such that $I\subset J$, and such that none of the functions $v_i$ vanish on $J$. Then  the functions $v_i$ define an invariant analytic metric on a cohomogeneity one manifold $M'=G/H\times J$, where $H$ is the  isotropy group at the regular points of $\gamma(t)$.
 Although this may not be the original metric on  $J\backslash I$, it is still curvature homogeneous. Indeed, the components of the curvature  are analytic functions, and  since they are constant on $I$, they must be constant on $J$ as well.   We choose $J$ to be maximal with this property and we then have two possible cases:

\begin{enumerate}
	\item[(a)]  $J\subsetneq \R$. This implies that at least one of the functions $v_i$ has a zero at a boundary point $t_0$ of $J$. We assume, for simplicity, $t_0=0$ and $J=(0,a)$. Moreover one of the functions, say $v_1$, takes on (up to scaling) one of the following simplified forms:
	$$v_1=a_1\,\sin(t),\qquad v_1=a_1\,(e^t-e^{-t}),\qquad  v_1=a_1\, t.$$
	The component of the curvature tensor on $M'$ must be constant in $J$. The vanishing of the derivatives of the curvature tensor  will enable us to determine the functions $v_i$.
	 Notice though that we are not allowed to use the smoothness conditions as we did in Section 5. 
	 \bigskip
	 
	\item[(b)] $J=\R$. In this case the functions $v_i$ fall into one of the following types
	$$v_i=a_i\, e^{d_i t}+b_i\, e^{-d_i t},\quad v_i=a_i\, e^{2 d_i t},\qquad v_i=a_i$$
	where $a_i$ and $b_i$ may be assumed to be positive, and $d_i>0$.  In this case it will be sufficient for us to use the fact that the sectional curvature $	R(E_2,E_3,E_2,E_3)$  is equal to a constant $k$. Thus  the numerator of $R(E_2,E_3,E_2,E_3)-k$ must be 0. It will turn out that these metrics will never be curvature homogeneous.
  \end{enumerate}

\bigskip
The result is that the functions $v_i$ in the interval $J$ and hence in $I$, agree with one of the known examples in Section 3, with one exception, but this exception does not extend to a  smooth metric on $M$.
The calculations though are significantly more involved than in Section 5. We illustrate the process with one example for each case.

	\bigskip

	For case (a) let
	$$v_1=a_1\, t,\qquad v_2=a_2 \sin(d_2 t)+b_2 \cos(d_2 t)\,\qquad v_3=a_3\, e^{d_3 t}+b_3\, e^{-d_3 t},$$
	with $a_1,d_2,d_3\neq 0$. The leading term in the power series of $R(E_2,E_3,E_2,E_3)$ is
	$$\frac{(a_3 + b_2 + b_3)^2(a_3 - b_2 + b_3)^2}{((a_3 + b_3)^2a1^2b2^2}t^{-2}.$$
	We assume first that $b_2\ne 0$, which implies that  $b_2=\pm (a_3+b_3)$. Let $b_2= (a_3+b_3)$, the other case being similar. Substituting into $R(E_1,E_3,E_1,E_3)$ and $R(E_1,E_2,E_1,E_2)$, their leading term becomes
	$$
	\frac{(a_1^2 - 8)(a_3 - b_3)d_3 + 8a_2d_2}{a_1^2(a_3 + b_3)}t^{-1}, \qquad
	\frac{  (a_1^2 - 8)a_2d_2 + 8d_3(a_3 - b_3)}{a_1^2(a_3 + b_3)}t^{-1}
	$$
	which easily implies that either $ a_2=0$ and $ a_3=b_3$, or $a_1=4$ and $(a_3-b_3)d_3+a_2d_2=0$.
	
	In the first case,
	substituting into $R(E_3,E_1,E_2,E_4)$ and $R(E_1,E_2,E_3,E_4)$, the first non-constant term is:
	$$
	\frac{-4(d_2^2 + d_3^2)^2a_3^2 + a_1^2(d_2^2 - 3d_3^2)}{8a_3^2a_1}\ t^2,
	\qquad   \frac{-4(d_2^2 + d_3^2)^2a_3^2 + a_1^2(3d_2^2 - d_3^2)}{8a_3^2a_1}\ t^2
	$$
	which implies that $d_2=d_3=0$, which is not allowed.
	
	\smallskip
	
	In the second case, $a_1=4$ and $(a_3-b_3)d_3+a_2d_2=0$, we solve for $a_2$ and substitute into $R(E_2,E_3,E_1,E_4)$, the first non-constant term is:
	$$
	\frac{3(d_2^2 + d_3^2)(a_3 - b_3)d_3}{2a_3 + 2b_3} \, t
	$$
	which implies $b_3=a_3$ and we are back in the first case.
	\smallskip
	
	If $b_2=0$, and hence $a_2\ne 0$, then $R(E_2,E_3,E_2,E_3)$ is constant only if  $b_3=-a_3$. Substituting into $R(E_2,E_3,E_1,E_4)$ and $R(E_3,E_1,E_2,E_4)$, the first non-constant term is:
	$$
	\frac{-2a2^2d_2^2 + 4a_3^2d_3^2 + 2a_1^2}{d_2^2a_1^2a_2^2}t^{-2},\qquad
	\frac{2a2^2d_2^2 + 4a_3^2d_3^2 - 2a_1^2}{d_2^2a_1^2a_2^2}t^{-2}
	$$
	and hence $a_3d_3=0$. Thus there are no solutions in this case as well.
	
	The case where all 3 functions are trigonometric is more complicated and is left to the reader.
	
	\vspace{10pt}
	
	 For case (b), let
	$$v_1=a_1\, e^{d_1 t}+b_1\, e^{-d_1 t},\quad v_2=a_2\, e^{ d_2 t},\quad v_3=a_3\, e^{ d_3 t}.$$
	 with $a_i,b_i,d_i>0$. Then   $R(E_2,E_3,E_1,E_4)$ is equal to:
	\begin{align*}
	\frac{  -2e^{-t(d_1 + d_2 + d_3)}  }{  a_2a_3(b_1+a_1e^{2d_1t})}\left(
	a_2^2(d_2 - d_3)e^{2t(d_1 + d_2)} - a_3^2(d_2 - d_3)e^{2t(d_1 + d_3)} + 2a_1b_1(d_2 + d_3)e^{2d_1t}\right.\\ \left. - a_1^2(2d_1 - d_2 - d_3)e^{4d1t} + (2d_1 + d_2 + d_3)b_1^2\right).
	\end{align*}
		This must be equal to a constant $k$, and hence the numerator of  $R(E_2,E_3,E_1,E_4)-k$ must vanish:
		\begin{align*}
	&  -a_1^2(2d_1 - d_2 - d_3)e^{t(3d_1 - d_2 - d_3)} + 2a_1b_1(d_2 + d_3)e^{t(d_1 - d_2 - d_3)} - a_3^2(d_2 - d_3)e^{t(d_1 - d_2 + d_3)} \\  
	&  + a_2^2(d_2 - d_3)e^{t(d_1 + d_2 - d_3)} + (2d_1 + d_2 + d_3)b_1^2e^{-t(d_1 + d_2 + d_3)} - ka_2a_3a_1e^{2d_1t} - ka_2a_3 b_1   =0.
	\end{align*}

	The exponential functions must all cancel and we can use the fact that exponentials with different exponents are linearly independent. Thus we need to distinguish the cases where some coefficient of the exponential functions vanish, or some exponents are equal to $0$, or  two or more of the exponents coincide. One easily checks that the exponent $d_1-d_2-d_3$ cannot be equal to any of the others. Thus it must vanish. Substituting $d_1=d_2+d_3$ into the equations, we get
		\begin{align*}
	&  - \left(a_1^2(d_2 + d_3) + ka_1a_2a_3\right)e^{2t(d_2 + d_3)} + 3b_1^2(d_2 + d_3)e^{-2t(d_2 + d_3)} \\  
	&   + a_2^2(d_2 - d_3)e^{2d_2t} - a_3^2(d_2 - d_3)e^{2d_3t} + (2d_2 + 2d_3)a_1b_1 - ka_2a_3b_1   =0.
	\end{align*}
	The coefficient of the second term is positive, and its exponent cannot be equal to any of the other exponents. Thus there are no solutions of this type.

	The case where two or all three functions $v_i$ are sums of two exponentials is more involved since the equations will contain many exponential functions. This case is left to the reader.

	\smallskip

	Thus we have shown that there are  no curvature homogeneous metrics when the action of $G$ has no singular orbits. 
	
		\smallskip
	
	If there are singular orbits, we conclude that there exists an interval $I$ on which the metric functions $v_i$ agree with one of the examples in Section 3, although this interval may not contain the singular orbit. There is one  exception though (under the assumption that $G=SU(2)$) where the functions are slightly more general. It has the form:
	$$
	v_1=a\sin(t),\quad v_2=a\cos(t),\quad v_3=a,  \quad 0\le t\le \pi/2.
	$$
	This metric is only smooth when $a=2$, in which case it is the metric in Example  5 on $\Sph^2\times\Sph^2$.   Notice though that when $a$ is an integer,  the metric is a smooth orbifold metric.  
	
	\bigskip

	 The following finishes the proof:
	\bigskip
	
	{\bf Proof of Theorem A}
	We saw that on any  interval $I=(a,b)$ where the eigenvalues have constant multiplicity, the metric agrees,  up to an equivariant diffeomorphism, with one of the known examples.
	 Let us first consider the case where the eigenvalues of $P_t$ on $I$ are all distinct. This means that on the interval $I$ the metric agrees with one of Example 1-5 in Section 3. Now observe that these examples have the following property. If the metric is defined on the interval  $[0,L]$ (possibly $L=\infty$)  then in fact   the metric  has distinct eigenvalues for all $0<t<L$.
		Let $(a,b)$ be a maximal connected interval  where our metric $P_t$ has 3 distinct eigenvalues.  At $t=a$ two of the eigenvalues of $P_t$ coincide by assumption, i.e. there exists a pair, say $\lambda_1,\lambda_2$  such that $\lambda_1(a)=\lambda_2(a)$. We want to show that $a=0 $, so assume that $a>0$.

	  There exists a constant $D>0$ such that for the 5 known examples $P_t^o=\diag(v_1^o,v_2^o,v_3^o)$ we have $|v_i^o(t)-v_j^o(t)|>D$ for $i\ne j$ and $t\in [a,b]$.  Now choose a small $\e>0$ such that  $|\lambda_1(t)-\lambda_2(t)|<D/2$ for $t\in (a,a+\e)$.
	    By \pref{normalizer}, the metric is equivariantly isometric to one where $P_t$ is diagonal in $(a+\e,b-\e)$.  Our proof above shows that on that interval $P_t$ agrees with one of the examples in Section 3.
	    Hence  $|v_i^o(a+\e)-v_j^o(a+\e)| =|\lambda_1(a+\e)-\lambda_2(a+\e)|<D/2$.    This is a contradiction and hence $a=0$. Similarly, it follows that $b=L$.
		
		If there is no interval with 3 distinct eigenvalues, we repeat the same  argument if there exists an interval with two distinct eigenvalues, as in  Examples 6,7,9,10,  or an interval with only one eigenvalue, as in Example 8. This finishes the Proof of Theorem A.
	
	\begin{rem*}
		In all cases, except for the Tsukada example or one of the five orbifold examples, we can also argue as follows. The remaining known examples, i.e. the metrics on a symmetric space, are Einstein. After we prove that the metric on the interval $I$ must be one of the known examples,  it follows that it is Einstein on $I$. Being curvature homogeneous implies that it is Einstein on $M$. This is an ODE along the geodesic, and by the uniqueness of solutions of this ODE starting at a point on $I$, the metric is isometric to a known example.  One can also show that the Tsukada example satisfies an ODE, but it turns out that this differential equation is significantly more complicated.
	\end{rem*}

 \providecommand{\bysame}{\leavevmode\hbox
to3em{\hrulefill}\thinspace}


\begin{thebibliography}{9999}

\bibitem[AA]{AA} A.V.~Alekseevsky and D.V.~Alekseevsky,
{\em $G$- manifolds with one dimensional orbit space,} Ad. in Sov.
Math. {\bf 8} (1992), 1--31.

\bibitem[AB]{AB}   M. Alexandrino and R. Bettiol,
{\em
Lie Groups and Geometric Aspects of Isometric Actions}, Springer 2015.

	\bibitem[Be]{Be} A. Besse, \emph{Einstein Manifolds},  Modern Surveys in Math., vol. 10, Springer-Verlag, 1987

\bibitem[BKV]{BKV} Boeckx, O. Kowalski, and L. Vanhecke, {\em Riemannian manifolds of
conullity two ,}  World Scientific Publishing Co., 1996, pp. xviii+300.

\bibitem[Br]{Br} T. Brooks, {\em Three dimensional manifolds with constant Ricci eigenvalues $(\lambda,\lambda,0)$,} in preparation.






\bibitem[FKM]{FKM} D.Ferus, H.Karcher and H.F. M\"unzner,
{\em Clifford algebras and new isoparametric hypersurfaces}, Math. Z. {\bf 177} (1981), 479-502.

\bibitem[GZ]{GZ} K.Grove-W.Ziller,
\emph{Lifting group actions and nonnegative curvature},
Trans. Amer. Math. Soc. 363 (2011) 2865-2890.

\bibitem[GVZ]{GVZ} K. Grove, L. Verdiani and W. Ziller, {\em  An exotic $T_1\Sph^4$ with positive curvature}, Geom. Funct. Anal. {\bf 21} (2011),
499-524.

\bibitem[L]{L} P.~Lax, {\em Linear Algebra and Its Applications,} 2nd. ed., Wiley-Interscience, (2007).

\bibitem[Pa]{Pa}
J.~Parker, \emph{$4$-dimensional ${G}$-manifolds with $3$-dimensional
	orbits}, Pacific J. Math. \textbf{125} (1986), 187--204.

\bibitem[Se]{Se}  K. Sekigawa, {\em On the Riemannian manifolds of the form $B_f  \times F^n$ ,}  Kodai
Math. Sem. Rep. 26 (1974/75), pp. 343--347.

\bibitem[Ta]{Ta}   H. Takagi,, {\em On curvature homogeneity of Riemannian manifolds,} Tohoku Math.J. 26 (1974),
581-585.

\bibitem[Si]{Si}  I. M. Singer, {\em Infnitesimally homogeneous spaces ,}  Comm. Pure Appl. Math.
13 (1960), pp. 685--697.

\bibitem[Ts]{Ts} K. Tsukada, {\em Curvature homogeneous hypersurfaces immersed in a real space form,}
Tohoku Math. J. 40 (1988), 221--244.



\bibitem[TV]{TV} F.~Tricerri, L.~Vanhecke, {\em Curvature homogeneous Riemannian manifolds,}
Ann. scient. Ec. Norm. Sup. {\bf 22} (1989), 535--554.


\bibitem[KTV]{KTV} 0. Kowalski, F.~Tricerri, L.~Vanhecke, {\em
Curvature homogeneous Riemannian manifolds,} J.
Math. Pures Appl. 71 (1992) 471-501.

\bibitem[V]{V} L.~Verdiani, {\em Curvature homogeneous metrics of cohomogeneity one,}
Riv. Mat. Univ. Parma {\bf 6} (1997), 179--200.

\bibitem[VZ]{VZ} L.~Verdiani,W.~Ziller, {\em Smoothness conditions in cohomogeneity one manifolds,}
Transf. Groups, DOI 10.1007/s00031-020-09618-9 (2020).

\bibitem[Zi]{Zi} W.Ziller,
\emph{On the geometry of cohomogeneity one manifolds with positive curvature},
in: Riemannian Topology and Geometric Structures
on Manifolds, in honor of Charles P.Boyer's 65th
birthday, Progress in Mathematics 271 (2009), 233--262.

\end{thebibliography}
\end{document}